\documentclass{article}

\usepackage{amsmath}
\usepackage{amsthm}
\usepackage{amssymb}
\usepackage{enumerate}
\usepackage{authblk}

\newtheorem{theorem}{Theorem}
\newtheorem{lem}[theorem]{Lemma}
\newtheorem{claim}[theorem]{Claim}
\theoremstyle{definition}
\newtheorem{defn}[theorem]{Definition}
\newtheorem{cor}[theorem]{Corollary}
\newtheorem{subclaim}[theorem]{Subclaim}
\theoremstyle{remark}
\newtheorem*{remark}{Remark}

\newcommand{\paren}[1]{\left( #1 \right)}
\newcommand{\brac}[1]{\left[ #1 \right]}
\newcommand{\curly}[1]{\left\{ #1 \right\}}
\newcommand{\forces}{\Vdash}

\DeclareMathOperator{\dom}{dom}
\DeclareMathOperator{\Lim}{Lim}
\DeclareMathOperator{\otp}{otp}
\DeclareMathOperator{\cf}{cf}
\DeclareMathOperator{\supp}{supp}
\DeclareMathOperator{\Col}{Col}

\title{The Special Aronszajn Tree Property 
	at $\kappa^+$ and $\square_{\kappa, 2}$}
	\author{John Susice\footnote{email: jpsusice@ucla.edu} 
	\footnote{This material is based upon work 
	supported by the National Science Foundation under Grants No. DMS-1363364 and DMS-1800613}}
\affil{Department of Mathematics, \\ University of California at Los Angeles}

\begin{document}

\maketitle{}

\begin{abstract}
	We show that for any regular cardinal $\kappa$, 
	$\square_{\kappa, 2}$ is consistent with ``all 
	$\kappa^+$-Aronszajn trees are special.'' 
	By a result of Shelah and Stanley \cite{shelah_stanley}
	this is optimal in the sense that 
	$\square_{\kappa, 2}$ may not be strengthened to 
	$\square_{\kappa}$. Using methods of Golshani and 
	Hayut \cite{golshani_hayut} we obtain our consistency 
	result simultaneously for all regular $\kappa$. 
\end{abstract}

\section*{Acknowledgments}
The author would like to thank Itay Neeman for his 
assistance in revising this paper. 

\section{Introduction}

Given some infinite cardinal $\kappa$, a \emph{$\kappa$-Aronszajn 
tree} is a tree of height $\kappa$ without cofinal branches all of 
whose levels have cardinality $< \kappa$. By classical results of 
K\"onig and Aronszajn, respectively, there are no $\aleph_0$-Aronszajn 
trees but there are $\aleph_1$-Aronszajn trees. 

Of particular interest to us are special Aronszajn trees. For any 
successor cardinal $\kappa^+$, we say that a $\kappa^+$-Aronszajn tree 
$T$ is \emph{special} if there exists a function 
$f \colon T \rightarrow \kappa$ 
such that if $x <_T y$ then $f(x) \neq f(y)$. 
Following \cite{golshani_hayut}, if there are Aronszajn trees of height $\kappa^+$ and 
all such trees are special, we say that the \emph{Special Aronszajn Tree Property} holds 
at $\kappa^+$, and denote this by $\mathsf{SATP}(\kappa^+)$. 

By a result of Baumgartner, Malitz, and Reinhardt \cite{baumgartner_malitz_reinhardt} the forcing 
axiom 
$\mathsf{MA}_{\aleph_1}$ implies $\mathsf{SATP}(\aleph_1)$. Laver and 
Shelah \cite{laver_shelah} showed that $\mathsf{SATP}(\aleph_2)$ is consistent assuming the 
existence of a weakly compact cardinal. The forcing which achieves this result is a Levy Collapse 
of $\kappa$ to $\aleph_2$ followed by an iteration of length $\geq \kappa^+$ of posets 
which successively specialize all new $\kappa$-Aronszajn trees arising in the extension. 

Golshani and Hayut \cite{golshani_hayut} showed that under the same assumption it is consistent that 
$\mathsf{SATP}(\aleph_1)$ and $\mathsf{SATP}(\aleph_2)$ hold simultaneously, and achieved a global 
result by showing that it is consistent that $\mathsf{SATP}(\kappa^+)$ holds simultaneously for all 
regular $\kappa$, assuming the existence of a proper class of supercompact cardinals.
This result is achieved by adapting the methods of \cite{laver_shelah} to specialize all possible 
names for trees of height $\kappa^+$ while anticipating the specialization of trees of height $\kappa$. 

Another class of combinatorial principles of interest to set theorists are \emph{square principles}. 
The original square principle $\square_{\kappa}$ was introduced by Jensen \cite{jensen}, who also 
introduced a weak variant $\square^{*}_{\kappa}$. Later Schimmerling \cite{combinatorial_principles_in_core_model}
investigated principles $\square_{\kappa, \lambda}$
of intermediate strength. Suppose $\kappa$ is an infinite cardinal and 
$\lambda$ is a nonzero (but potentially finite) cardinal. A $\square_{\kappa, \lambda}$ sequence is 
a sequence $\vec{\mathcal{C}} = \langle \mathcal{C}_{\alpha} \colon \alpha \in  \Lim{(\kappa^+)} \rangle$
such that: 
\begin{enumerate}[(1)]
	\item For all $\alpha \in \Lim{(\kappa^+)}$, $1 \leq | \mathcal{C}_{\alpha} | \leq \lambda$. 
	\item For all $\alpha \in \Lim{(\kappa^+)}$ and $C \in \mathcal{C}_{\alpha}$, 
		$C$ is club in $\alpha$ and $\otp{C} \leq \kappa$. 
	\item For all $\alpha \in \Lim{(\kappa^+)}$, every $C \in \mathcal{C}_{\alpha}$ \emph{threads}
		$\langle \mathcal{C}_{\beta} \colon \beta < \alpha \rangle$ in the sense that 
		$C \cap \beta \in \mathcal{C}_{\beta}$ for all $\beta$ which are limit points of $C$. 
\end{enumerate}
We say $\square_{\kappa, \lambda}$ holds if such a sequence exists. Jensen's original principle 
$\square_{\kappa}$ is $\square_{\kappa, 1}$, and the weak square principle $\square^*_{\kappa}$ 
is $\square_{\kappa, \kappa}$. In this paper we will be concerned exclusively with 
$\square_{\kappa, 2}$.

We show that the result of Laver and Shelah may be improved by establishing the consistency 
of $\mathsf{SATP}(\aleph_2)$ with $\square_{\omega_1, 2}$. By a result of Shelah and Stanley 
\cite{shelah_stanley} $\mathsf{SATP}(\aleph_2)$ is incompatible with $\square_{\omega_1}$, so this 
result is optimal. Our result is obtained by using an iteration similar to that of Laver-Shelah, 
with the exception that we use a poset of Cummings and Schimmerling \cite{indexed_squares}--which collapses 
weakly compact $\kappa$ to $\aleph_2$ while adding a $\square_{\omega_1, 2}$-sequence simultaneously--in 
place of the Levy Collapse.

Furthermore, we show that our methods are compatible with the anticipatory framework of 
Golshani and Hayut, and thus we are also able to obtain the analogous global result--namely the 
consistency of $\mathsf{SATP}(\kappa^+)$ plus $\square_{\kappa, 2}$ for all regular $\kappa$.

\section{The Cummings-Schimmerling Poset}

We describe here a poset introduced by Cummings and Schimmerling 
\cite{indexed_squares}
which, given cardinals $\mu < \kappa$ with $\mu$ regular and 
$\kappa$ inaccessible, will simultaneously collapse 
$\kappa$ to become $\mu^+$ and add a $\square_{\mu, 2}$ sequence. 

We will denote this poset by $\mathbb{P}(\mu, < \kappa)$. 
A conditiion $p$ in $\mathbb{P}(\mu, < \kappa)$ is a function 
such that: 

\begin{enumerate}
	\item The domain of $p$ is a closed set of ordinals below $\kappa$
		of cardinality $< \mu$. 
	\item If $\alpha \in \dom{p}$ is a successor ordinal, say $\alpha = \bar{\alpha} + 1$, 
		then the unique element of $p(\alpha)$ is $\left\{ \bar{\alpha} \right\}$. 
\item If $\alpha \in \dom{p}$ is a limit ordinal with cofinality $< \mu$ then 
	$1 \leq | p(\alpha) | \leq 2$ and each element of 
	$p(\alpha)$ is a club subset of $\alpha$ with order type 
	$< \mu$. 
\item If $\alpha \in \dom{p}$ is a limit ordinal with cofinality $\geq \mu$ then 
	$p(\alpha) = \curly{C}$, where $C$ is some closed subset 
	of $\alpha$ with order type $< \mu$ such that 
	$\max{C} = \sup{(\dom{p} \cap \alpha)}$. 
\item (Coherence) If $\alpha \in \dom{p}$, $C \in p(\alpha)$, 
	and $\beta \in C$, then $\beta \in \dom{p}$. If moreover 
	$\beta \in \Lim{C}$, then 
	$C \cap \beta \in p(\beta)$. 
\end{enumerate}

The ordering of the poset is defined by $q \leq p$ iff 
$\dom{q} \supseteq \dom{p}$ and:

\begin{enumerate}[(a)]
	\item $q(\alpha) = p(\alpha)$ for all $\alpha \in \dom{p}$ of 
		cofinality $< \mu$.
	\item If $\alpha$ is of cofinality $\geq \mu$, 
		$p(\alpha) = \curly{C}$, and 
		$q(\alpha) = \curly{D}$, then 
		$C = D \cap (\max{C} + 1)$. 
\end{enumerate}

The Cummings-Schimmerling poset will not be $< \mu$-closed but will still be 
sufficiently closed so as to not add bounded subset of $\mu$: 

\begin{defn}
	Suppose that $\nu$ is some ordinal. A poset $\mathbb{P}$ is said to be 
	\emph{$\nu$-strategically closed} if Player II has a winning strategy in the 
	following game $G(\mathbb{P}, \nu)$ of length $\nu$: 
	\begin{center}
	\begin{tabular}{c | c c c c c c c c c }
		I &  $p_1$ & & $p_3$ & $\cdots$ & & $p_{\omega + 1}$ & & $\cdots$ \\ \hline
		II & & $p_2$ & & $\cdots$ & $p_{\omega}$ & & $p_{\omega + 2}$ & $\cdots$
	\end{tabular}
\end{center}
In this game the two players alternate building a descending chain $\langle p_{\xi} \colon 1 \leq \xi < \nu \rangle$ 
with Player II playing at all even ordinals (including limits) and Player II loses if he is unable to make a legal move.  

Suppose that $\mu$ is some cardinal. We say that $\mathbb{P}$ is \emph{$< \mu$-strategically closed} if 
it is $\nu$-strategically closed for all $\nu < \mu$. 

\end{defn}

\begin{lem}
	Suppose that $\mu < \kappa$ are cardinals with 
	$\mu$ regular and $\kappa$ inaccessible. 
	Then $\mathbb{P}(\mu, < \kappa)$ is $< \mu$-strategically closed and 
	$\kappa$-Knaster. 
	 Moreover, in the generic 
	extension by $\mathbb{P}(\mu, < \kappa)$, 
	$\kappa = \mu^+$ and $\square_{\mu, 2}$ holds. 
	\label{basic_properties_of_p_lem}
\end{lem}

\begin{proof}
	We prove strategic closure. Fix $\nu < \mu$. We define 
	a winning strategy for Player II in the game $G(\mathbb{P}(\mu, < \kappa), \nu)$. 
	Suppose that $\xi < \nu$ is even and $\langle p_{\zeta} \colon 1 \leq \zeta < \xi \rangle$
	have already been played. 
	
	If $\xi$ is a successor ordinal then Player II plays 
	an arbitrary extension $p_{\xi}$ of $p_{\xi - 1}$ such that 
	$\sup{(\dom{(p_{\xi})})}$ is strictly greater than 
	$\sup{(\dom{(p_{\xi - 1})})}$.  
	Now suppose that $\xi$ is a nonzero limit ordinal and let 
	\begin{align*}
		D = \bigcup_{1 \leq \zeta < \xi}{\dom{(p_{\xi})}}
	\end{align*}
	We define a condition $p_{\xi}$ with domain $\bar{D} = D \cup \Lim{(D)}$
	which extends all $\langle p_{\zeta} \colon 1 \leq \zeta < \xi \rangle$. 
	First, if $\alpha \in D$ with $\cf{(\alpha)} < \mu$, we let 
	$p_{\xi}(\alpha) = p_{\zeta}(\alpha)$ for any $1 \leq \zeta < \xi$ such that 
	$\alpha \in \dom{(p_\zeta)}$. 
	If $\alpha \in D$ with $\cf{(\alpha)} \geq \mu$ then for each 
	$1 \leq \zeta < \xi$ let $C_{\zeta}$ be the unique element of 
	$p_{\alpha}(\zeta)$ and let $p_{\xi}(\alpha) = \curly{C}$, where 
	\begin{align*}
		C = \bigcup_{\zeta < \xi}{C_{\zeta}} \cup \bigg\{ \sup{ \bigcup_{\zeta < \xi}{C_{\zeta}}} \bigg\}
	\end{align*}
	If $\alpha \in \Lim{D} \setminus D$ is below $\sup{D}$ then 
	choose $\beta \in D$ least such that $\alpha < \beta$. Note that $\beta$ must 
	be a limit ordinal by conditions (2), (5) in the definition of $\mathbb{P}(\mu, < \kappa)$. We claim 
	$\cf{(\beta)} \geq \mu$. 

	Suppose otherwise and let $E$ be an element of $p_{\xi}(\beta)$ (note that 
	$p_{\xi}(\beta)$ has already been defined). If $\alpha$ is not a limit point of 
	$E$, then let $\gamma$ be the least element of $E$ above $\alpha$. Then 
	$\alpha < \gamma < \beta$, and by condition (5) in the definition of $\mathbb{P}(\mu, < \kappa)$
	we have $\gamma \in D$, contradicting choice of $\beta$. 

	Thus $\cf{(\beta)} \geq \mu$ as desired, and so if we let 
	$E$ be the unique element of $p_{\xi}(\beta)$ (again, this has already been defined)
	then clause (4) in the definition of $\mathbb{P}(\mu, < \kappa)$ guarantees 
	$\max{(E)} = \alpha$, and we may define $p_{\xi}(\alpha) = \curly{E}$. 

	Finally, if $\alpha = \sup{D}$, let 
	\begin{align*}
		p_{\xi}(\alpha) = 
		\left\{ \max{(\dom{(p_{\zeta})})} \colon 1 \leq \zeta < \xi \right\}
	\end{align*}
	It should be clear that $p_{\xi}$ as defined above is a condition in 
	$\mathbb{P}(\mu, < \kappa)$ and the strategy described is a winning 
	strategy for Player II in $G(\mathbb{P}(\mu, < \kappa), \nu)$. 

	The rest of the lemma may be proved exactly as in \cite{indexed_squares}. 
\end{proof}

Note that an argument similar to the one above will show that if $\mu = \aleph_1$ then 
$\mathbb{P}(\mu, <\kappa)$ is in fact countably closed.

In the proof of our consistency results it will be crucial 
that for $\mu < \kappa_0 < \kappa_1$ with $\mu$ regular and 
$\kappa_0$, $\kappa_1$ inacessible, $\mathbb{P}(\mu, < \kappa_0)$
may be viewed as a factor of $\mathbb{P}(\mu, < \kappa_1)$. 
In order to precisely state the necessary factorization result, 
we first define two auxilliary posets: 

Suppose that $\vec{\mathcal{C}} = \langle \mathcal{C}_\alpha \colon \alpha < \mu^+ \rangle$ 
is a $\square_{\mu, 2}$-sequence. We let $\mathbb{T} = \mathbb{T}_{\vec{\mathcal{C}}}$ be the poset of 
closed bounded $C \subseteq \mu^+$ of order-type $< \mu$ such that 
$C$ threads $\langle \mathcal{C}_{\alpha} \colon \alpha \leq \max{C} \rangle$ in the sense 
that $C \cap \alpha \in \mathcal{C}_{\alpha}$ for all $\alpha \in \Lim{C}$. 
For $C, D \in \mathbb{T}$, we set $D \leq C$ if and only if $D$ is an end-extension of 
$C$. 

Finally, if $G$ is the generic added by $\mathbb{P}(\mu, < \kappa_0)$, 
then $\mathbb{Q} = \mathbb{Q}_{\mu, \kappa_0, \kappa_1, G}$ is the poset 
defined in $V\brac{G}$ by setting $q \in \mathbb{Q}$ iff: 

\begin{enumerate}[(1)]
	\item The domain of $q$ is a set of limit ordinals in the interval 
		$(\kappa_0, \kappa_1)$ of size $< \mu$. 
	\item If $\alpha \in \dom{q}$ has cofinality $< \mu$ then 
	$1 \leq | q(\alpha) | \leq 2$ and each element of 
	$q(\alpha)$ is a club subset of $\alpha$ with order type 
	$< \mu$. 
\item If $\alpha \in \dom{q}$ has cofinality $\geq \mu$ then 
	$q(\alpha) = \curly{C}$, where $C$ is a club subset of 
	$\alpha$ with order type $< \mu$ such that 
	$\max{C} \geq \sup{( \dom{q} \cap \alpha)}$.
\item (Coherence) If $\alpha \in \dom{q}$, $C \in q(\alpha)$, and 
	$\beta \in \Lim{C}$, then: 
	\begin{enumerate}[(A)]
		\item If $\beta > \kappa_0$, then $\beta \in \dom{q}$ and $C \cap \beta \in q(\beta)$. 
		\item If $\beta < \kappa_0$, then $C \cap \beta \in \mathcal{C}_\beta$, where 
			$\langle \mathcal{C}_\beta \colon \beta < \kappa_0 \rangle$ is $\bigcup{G}$. 
	\end{enumerate}
\end{enumerate}
For two elements $p, q \in \mathbb{Q}_{\mu, \kappa_0, \kappa_1, G}$, we set $p \leq q$ 
iff: 
\begin{enumerate}[(1)]
	\item $\dom{q} \subseteq \dom{p}$. 
	\item For all $\alpha \in \dom{q}$: 
		\begin{enumerate}[(a)]
			\item If $\alpha$ has cofinality $< \mu$ then $p(\alpha) = q(\alpha)$. 
			\item If $\alpha$ has cofinality $\geq \mu$, $p(\alpha) = \curly{C}$, and 
				$q(\alpha) = \curly{D}$, then $C$ is an end-extension of 
				$D$. 
		\end{enumerate}
\end{enumerate}

\begin{lem}
	Suppose that $\mu < \kappa_0 < \kappa_1$ are cardinals with $\mu$ regular and $\kappa_0$, 
	$\kappa_1$ inaccessible, and $\dot{G}$ is the canonical name for the 
	$\mathbb{P}(\mu, < \kappa_0)$-generic. Then if we let $\dot{\mathbb{T}} = \check{\mathbb{T}}_{
		\bigcup{\dot{G}}}$, $\dot{\mathbb{Q}} = \check{\mathbb{Q}}_{\mu, \kappa_0, \kappa_1, \dot{G}}$, there is 
		an isomorphism between a dense subset of $\mathbb{P}(\mu, < \kappa_1)$ and a dense 
		subset of $\mathbb{P}(\mu, < \kappa_0) \ast \dot{\mathbb{T}} \ast \dot{\mathbb{Q}}$. 
		In particular these two forcings are equivalent, and so informally we may view them as 
		being equal. 
\end{lem}

\begin{proof}
	As in \cite{indexed_squares}. 
\end{proof}

\begin{lem}
	Suppose that $\mu < \kappa$ are cardinals with $\mu$ regular and $\kappa$ inaccessible. 
	Let $\mathbb{P} = \mathbb{P}(\mu, < \kappa)$, let $\dot{G}$ be the canonical name for 
	the $\mathbb{P}$-generic, and let $\dot{\mathbb{T}} = \check{\mathbb{T}}_{\bigcup \dot{G}}$. 
	Then there is a dense subset of $\mathbb{P} \ast \dot{\mathbb{T}}$ which is 
	$< \mu$-closed and so in particular $\mathbb{P} \ast \dot{\mathbb{T}}$ is 
	forcing equivalent to $\Col{(\mu, \kappa)}$. 
	\label{square_star_threading_lem}
\end{lem}

\begin{proof}
	Let $D$ be the dense set of conditions in $\mathbb{P} \ast \dot{\mathbb{T}}$ of the form 
	$(p, \check{t})$ which are \emph{flat} in the sense that $\max{(\dom{p})} = \max{t}$. 
	We claim that $D$ is as desired. To see this, suppose that $\nu < \mu$ is a limit 
	ordinal and let 
	$\langle (p_{\xi}, \check{t}_\xi) \colon \xi < \nu \rangle$ 
	be a descending sequence of conditions in $D$. We find a lower bound 
	$p^*$ for $\langle p_{\xi} \colon \xi < \nu \rangle$ as in the 
	limit case of Lemma \ref{basic_properties_of_p_lem}, except that we 
	set 
	\begin{align*}
		p^*\big( \sup_{\xi < \nu}{(\max{( \dom{p_{\xi}})} )} \big) = \curly{t^*}
	\end{align*}
	where 
	\begin{align*}
		t^* = \bigcup_{\xi < \nu}{t_{\xi}}
	\end{align*}
	Then $(p^*, \check{t}^*) \in D$ is our desired lower bound. Since $D$ is 
	$< \mu$-closed, $|D| = \kappa$, and $D$ forces $|\kappa| = \mu$, $D$ 
	is forcing equivalent to $\Col{(\mu, \kappa)}$ by a well-known result due to Solovay 
	(see, e.g. Lemma 2.3 of \cite{cc_weak_square}). 
\end{proof}

\section{Specializing Trees with Anticipation}
In this section we review the methods of \cite{golshani_hayut} for specializing trees 
while anticipating subsequent forcing. 

First we introduce the modified Baumgartner forcing which specializes a single tree 
while anticipating a single subsequent forcing. 

\begin{defn}[\cite{golshani_hayut}]
	Suppose that $\mu < \kappa$ are regular cardinals in $V$ and $\mathbb{I}^{2} \ast \dot{ \mathbb{I}}^1$ is a 
	$\kappa$-c.c. two-step iteration which forces $\kappa = \mu^+$. Suppose moreover that $\dot{T}$ is 
	an $\mathbb{I}^2 \ast \dot{\mathbb{I}}^1$-name for a $\kappa$-Aronszajn tree, which 
	we view as a subset of $\kappa \times \mu$. Then $\mathbb{B}_{\mu, \mathbb{I}^1}\!(\dot{T})$
	is defined in $V^{\mathbb{I}^2}$ as the poset of partial functions 
	$f \colon \kappa \times \mu \rightarrow \mu$ of size $< \mu$ such that if 
	$s, t \in \dom{f}$ and $f\!\paren{s} = f\!\paren{t}$, then 
	\begin{align*}
		\forces^{V^{\mathbb{I}^2}}_{\mathbb{I}^1} \check{s} \perp_{\dot{T}} \check{t}
	\end{align*}	
	The forcing is ordered by reverse inclusion. 
\end{defn}
If $\mu$ is understood (as it usually is) then we suppress the dependence on $\mu$ and write 
$\mathbb{B}_{\mathbb{I}^1}(\dot{T})$ in place of $\mathbb{B}_{\mu, \mathbb{I}^1}(\dot{T})$. 

\begin{lem}[\cite{golshani_hayut}]
	Suppose $\mu < \kappa$ are regular cardinals and, $\mathbb{I}^2 \ast \dot{\mathbb{I}}^1$ is 
	a $\kappa$-c.c. forcing which forces $\kappa = \mu^+$, $\dot{T}$ is an $\mathbb{I}^2 \ast 
	\mathbb{I}^1$-name for a $\kappa$-Aronszajn tree, and $G$ is $\mathbb{I}^2$-generic. 
	Then in $V\!\brac{G}$ the following 
	hold: 
	\begin{enumerate}[(a)]
	\item $\mathbb{B}_{\mathbb{I}^1}(\dot{T})$ is $< \mu$-closed. 
	\item In the extension by the generic for $\mathbb{B}_{\mathbb{I}^1}(\dot{T})$ there is 
		a function $F \colon \kappa \times \mu \rightarrow \mu$ which is 
		a specializing function for the tree $\dot{T}\brac{G}\brac{H}$ for \emph{any}
		$\mathbb{I}^1$-generic $H$. 
	\end{enumerate}
\end{lem}

Now we describe the general form of iterations $\vec{\mathbb{I}}^2$ and $\vec{\mathbb{I}}^1$ such that 
$\vec{\mathbb{I}}^2$ specializes \emph{all} $\kappa$-Aronszajn trees while anticipating forcing 
by $\vec{\mathbb{I}}^1$. 

\begin{defn}
	Suppose that $\mu < \kappa < \kappa^+ \leq \delta$ are regular cardinals in $V$, and the iterations 
\begin{align*}
	\vec{\mathbb{I}}^2 &= \langle \langle \mathbb{I}^2_{\gamma} \colon \gamma \leq \delta \rangle, 
	\langle \dot{\mathbb{J}}^{2}_{\gamma} \colon \gamma < \delta \rangle \rangle \\
	\dot{\vec{\mathbb{I}}}^1 &= \langle \langle \dot{\mathbb{I}}^1_{\gamma} \colon \gamma \leq \delta \rangle, 
	\langle \dot{\mathbb{J}}^{1}_{\gamma} \colon \gamma < \delta \rangle \rangle 
\end{align*}
are as follows: 

\begin{itemize}
	\item $\mathbb{I}^2_1 = \mathbb{P}(\mu, < \kappa)$, the forcing which collapses 
		$\kappa$ to $\mu^+$ while adding $\square_{\mu, 2}$. 
	\item $\mathbb{I}^2_{\gamma}$ is the iteration with $< \mu$-support 
		of $\langle \dot{\mathbb{J}}^{2}_{\gamma'} \colon \gamma' < \gamma \rangle$. In other 
		words, if $\gamma$ is a limit ordinal of cofinality $\geq \mu$, then 
		$\mathbb{I}^2_{\gamma}$ is the direct limit of $\langle \mathbb{I}^2_{\gamma'} \colon 
		\gamma' < \gamma \rangle$, if $\gamma$ is a limit ordinal of cofinality $< \mu$, then 
		$\mathbb{I}^2_{\gamma}$ is the inverse limit of $\langle \mathbb{I}^2_{\gamma'} \colon 
		\gamma' < \gamma \rangle$, and if $\gamma = \bar{\gamma} + 1$ is a successor ordinal then 
		$\mathbb{I}^2_{\gamma} = \mathbb{I}^2_{\bar{\gamma}} \ast \dot{\mathbb{J}}^2_{\bar{\gamma}}$. 
	\item Each $\dot{\mathbb{I}}^1_{\gamma}$ is an $\mathbb{I}^2_{\gamma}$-name for a 
		$\mu$-c.c. poset. 
	\item $\dot{\mathbb{J}}^2_{\gamma}$ is a name for the poset 
		$\mathbb{B}_{\mathbb{I}^1_{\gamma}}(\dot{T}_\gamma)$, where 
		$\dot{T}_{\gamma}$ is an $\mathbb{I}^2_\gamma \ast \dot{\mathbb{I}}^1_{\gamma}$-name 
		for a $\kappa$-Aronszajn tree, chosen according to some appropriate bookkeeping function. 
\end{itemize}
Then we refer to $\vec{\mathbb{I}}^2$ as an ``iteration which collapses $\kappa$ to $\mu^+$, adds $\square_{\mu, 2}$ and 
specializes all $\kappa$-Aronszajn trees while
anticipating the subsequent iteration $\vec{\mathbb{I}}^1$'' (or some similar locution for the sake of 
brevity). 

\end{defn}

\begin{defn}[\cite{golshani_hayut}]
	Suppose that $\mu < \kappa < \delta$ are regular cardinals and $\vec{\mathbb{I}}^2$ is an 
	iteration of length $\delta$ which collapses $\kappa$ to $\mu^+$, adds $\square_{\mu, 2}$, 
	and anticipates the iteration $\dot{\vec{\mathbb{I}}}^1$ in the sense described above. 
	Let $\gamma \leq \delta$ be some ordinal and 
	suppose that $M$ is an elementary substructure of $H(\theta)$ ($\theta$ sufficiently large)
	of cardinality $\kappa$ such that $V_{\kappa} \cup M^{< \kappa} \cup \curly{\gamma} \subseteq M$ and 
	$M$ contains all relevant parameters. 
	Furthermore, let $\phi \colon \kappa \rightarrow M$ be a 
	bijection and for all $\alpha < \kappa$ set $M_\alpha = \phi `` \alpha$. We say that 
	$\vec{\mathbb{I}}^2$, $\dot{\vec{\mathbb{I}}}^1$ are \emph{suitable for $M$, $\phi$}, $\gamma$ if: 
	\begin{enumerate}[(1)]
		\item For all $\bar\gamma \leq \gamma$, 
			$\forces_{\mathbb{I}^2_{\gamma}}$ ``$\dot{\mathbb{I}}^1_{\bar{\gamma}}$ is 
			$\mu$-c.c.'' 
		\item For all $\alpha < \kappa$ and $\bar{\gamma} \in M_\alpha \cap \gamma$, if: 
			\begin{enumerate}[(a)]
				\item $\mathbb{I}^2_{\bar\gamma} \cap M_\alpha$ is a regular subposet of 
					$\mathbb{I}^2_{\bar\gamma} \cap M$. 
				\item $\forces_{\mathbb{I}^2_{\bar\gamma} \cap M}$ ``$\dot{\mathbb{I}}^1_{\bar\gamma} \cap 
					M_\alpha$ is a regular subiteration of $\dot{\mathbb{I}}^1_{\bar\gamma} \cap M$.''
				\item $\dot{T}_{\bar\gamma} \cap M_\alpha$ is an $(\mathbb{I}^2_{\bar\gamma} \cap M_\alpha) 
					\ast (\dot{\mathbb{I}}^1_{\bar\gamma} \cap M_\alpha)$-name for an 
					$\alpha$-Aronszajn tree. 
			\end{enumerate}
			Then forcing with $(\mathbb{I}^2_{\bar\gamma} \cap M) \ast (\dot{\mathbb{I}}^1_{\bar\gamma} 
			\cap M) / G$, where $G$ is generic for 
			$(\mathbb{I}^2_{\bar\gamma} \cap M) \ast (\dot{\mathbb{I}}^1_{\bar\gamma} \cap M_\alpha)$,
			doesn't add any new branches to the tree named by $\dot{T}_{\bar\gamma} \cap M_\alpha$. 
	\end{enumerate}
	\label{suitability_defn}
\end{defn}

We will need to make use of the following basic lemma about forcings which don't add branches to trees: 

\begin{lem}[Folklore, see \cite{cummings_foreman}, \cite{kunen_tall}]
	Suppose that $T$ is a $\kappa$-tree and $\mathbb{P}$ is a 
	$\kappa$-Knaster poset. Then forcing with $\mathbb{P}$ doesn't 
	add a branch to $T$. 
	\label{doesnt_add_branch_lem}
\end{lem}

\section{Obtaining $\square_{\omega_1, 2}$ + $\mathsf{SATP}(\aleph_2)$ + 
$\mathsf{SATP}(\aleph_1)$}

\begin{theorem}
	Suppose that $\mu < \kappa < \kappa^+ \leq \delta$ are cardinals with 
	$\mu$, $\delta$ regular and $\kappa$ weakly compact. 
	Suppose moreover that 
	\begin{align*}
		\vec{ \mathbb{I}}^2 &= \langle \langle \mathbb{I}^2_{\gamma} \colon \gamma \leq \delta 
		\rangle, \langle \dot{\mathbb{J}}^2_{\gamma} \colon \gamma < \delta \rangle \rangle \\
		\dot{\vec{\mathbb{I}}}^1 &= 
		\langle \langle \dot{\mathbb{I}}^1_{\gamma} \colon \gamma \leq \delta \rangle, 
		\langle \dot{\mathbb{J}}^1_{\gamma} \colon \gamma < \delta \rangle \rangle 
	\end{align*}
	are two iterations such that $\vec{\mathbb{I}}^2$ collapses $\kappa$ to $\mu^+$, 
	adds $\square_{\mu, 2}$, and specializes all $\kappa$-Aronszajn trees while 
	anticipating $\vec{\mathbb{I}}^1$ (in the sense described in the previous section). 

	Finally, suppose that for all ordinals $\gamma \leq \delta$
	there exists $M$ elementary in $H(\theta)$ ($\theta$ sufficiently large)
	of cardinality $\kappa$ such that $V_\kappa \cup M^{< \kappa} \cup \curly{\gamma} 
	\subseteq M$, $M$ contains 
	all relevant parameters, and $\vec{\mathbb{I}}^2$, $\dot{\vec{\mathbb{I}}}^1$ are suitable 
	for $M$, $\phi$, $\gamma$ (for some fixed bijection $\phi \colon \kappa \rightarrow M$). 
	Then the generic extension by $\mathbb{I}^2_\delta \ast \dot{\mathbb{I}}^1_{\delta}$ 
	satisfies
	\begin{align*}
		\kappa = \mu^+ \land \square_{\mu, 2} 
		\land \mathsf{SATP}(\kappa) \land 2^{\mu} \geq \delta
	\end{align*}
	\label{main_theorem}
\end{theorem}
The majority of the remainder of this section is devoted to giving a proof of this result. 
We follow closely the proof of the main theorem in \cite{golshani_hayut}.

\begin{lem}
	For every $\gamma \leq \delta$, $\mathbb{I}^2_{\gamma}$ is $< \mu$ strategically 
	closed. 
\end{lem}

\begin{proof}
	The forcing $\mathbb{I}^{2}_{\gamma}$ is a $< \mu$-strategically closed forcing 
	(namely, $\mathbb{P}(\mu, < \kappa)$) followed by the $< \mu$-support iteration of 
	$< \mu$-closed posets. 
\end{proof}

\begin{lem}
	For every $\gamma \leq \delta$, $\forces_{\mathbb{I}^2_\gamma} ``\dot{\mathbb{I}}^1_{\gamma}$ is 
	$\mu$-c.c.'' and $\forces_{\mathbb{I}^2_\delta} `` \dot{\mathbb{I}}^1_{\gamma}$ is $\mu$-c.c.''
	\label{chain_condition_of_I1_lem}
\end{lem}

\begin{proof}
	This is immediate from the definition of suitability of $\vec{\mathbb{I}}^2$, $\dot{\vec{\mathbb{I}}}^1$. 
\end{proof}

\begin{lem}
	For every $\gamma \leq \delta$, $\mathbb{I}^{2}_{\gamma}$ is $\kappa$-Knaster. 
	\label{chain_condition_lem}
\end{lem}

\begin{proof}
	By induction on $\gamma$. For the base case, we know $\mathbb{I}^{2}_{1} \simeq 
	\mathbb{P}\!\paren{\mu, < \kappa}$ is $\kappa$-Knaster by Lemma \ref{basic_properties_of_p_lem}. So suppose $\gamma \leq \delta$
	and each $\mathbb{I}^{2}_{\gamma'}$ is $\kappa$-Knaster for all $\gamma' < \gamma$. 
	We seek to show that $\mathbb{I}^{2}_{\gamma}$ is also $\kappa$-Knaster. 

	If $\gamma$ is a limit ordinal and $\mu \leq \cf{\gamma} \neq \kappa$
	this is immediate since any subset of $\mathbb{I}^{2}_{\gamma}$ of cardinality 
	$\kappa$  may be refined 
	to a subset of $\mathbb{I}^2_{\gamma'}$ of cardinality $\kappa$ for some $\gamma' < \gamma$. 

	If $\gamma$ is a limit ordinal with $\cf{\gamma} = \kappa$ this follows from a
	$\Delta$-system argument. 

	Thus suppose that either $\gamma$ is a limit ordinal with 
	$\cf{\gamma} < \mu$ or $\gamma = \bar{\gamma} + 1$ for 
	some ordinal $\bar{\gamma}$. Fix $M$ as in the hypothesis of the theorem, and in 
	either case fix an increasing sequence 
	$\left\{ \gamma_i \colon i < \cf{\gamma} \right\}$  in $M$ which is cofinal in $\gamma$ (if 
	$\gamma = \bar{\gamma} + 1$ is a successor ordinal we say its cofinality is $1$ and we let 
	$\gamma_0 = \bar{\gamma}$, so in this case $\left\{ \gamma_i \colon i < 1 \right\}$ is cofinal in $\gamma$).

	Let $R$ be a subset of $V_\kappa$ which encodes both $M$ and $\phi$ (where $\phi$ is the bijection from 
	the hypothesis of Theorem \ref{main_theorem}). 
	Fix a $< \kappa$-complete normal filter $\mathcal{F}$ on $\kappa$ which extends the club 
	filter and satisfies 
	\begin{align*}
		\left\{ \alpha < \kappa \colon \paren{V_{\alpha}, \in, R \cap V_\alpha} \models 
	\psi \right\} \in \mathcal{F}
	\end{align*}
	for each formula $\psi$ which is $\Pi^1_{1}$ over $V_\kappa$. 
For all $\alpha < \kappa$ set 
$M_\alpha = \phi `` \alpha$.

\begin{claim}[\cite{laver_shelah}, \cite{golshani_hayut}]
	Assume that $\bar{\gamma} \in \gamma \cap M$, and for all 
	$\bar{\bar{\gamma}} \in \bar{\gamma} \cap M$ we have that $\mathbb{I}^2_{\bar{\bar{\gamma}}}$ is $\kappa$-Knaster and 
	$\dot{T}_{\bar{\bar{\gamma}}}$ is an $\mathbb{I}^2_{\bar{\bar{\gamma}}} \ast \dot{\mathbb{I}}^1_{\bar{\bar{\gamma}}}$-name 
	for a $\kappa$-Aronszajn tree. Then there exists $X = X_{\bar{\gamma}} \in \mathcal{F}$ such that for all 
	$\alpha \in X$ and $\bar{\bar{\gamma}} \in \bar{\gamma} \cap M_{\alpha}$: 
	\begin{enumerate}
		\item $\alpha$ is inaccessible. 
		\item $M_{\alpha} \cap \kappa = \alpha$. 
		\item $M_{\alpha}^{< \alpha} \subseteq M_\alpha$. 
		\item $\mathbb{I}^{2}_{\bar{\bar{\gamma}}} \cap M_\alpha$ is a regular subposet of 
			$\mathbb{I}^2_{\bar{\bar{\gamma}}} \cap M$
			and is $\alpha$-c.c. 
		\item $\mathbb{I}^1_{\bar{\bar{\gamma}}} \cap M_{\alpha}$ is equivalent to an
			$\mathbb{I}^{2}_{\bar{\bar{\gamma}}} \cap M_\alpha$-name. 
		\item $(\mathbb{I}^2_{\bar{\bar{\gamma}}} \ast \dot{\mathbb{I}}^1_{\bar{\bar{\gamma}}}) \cap M_\alpha$ is a 
			regular subposet of 
			$( \mathbb{I}^2_{\bar{\bar{\gamma}}} \ast \dot{\mathbb{I}}^1_{\bar{\bar{\gamma}}}) 
			\cap M$. 
		\item $(\mathbb{I}^2_{\bar{\bar{\gamma}}} \ast \dot{\mathbb{I}}^1_{\bar{\bar{\gamma}}}) \cap 
			M_\alpha$ forces that 
			$T_{\bar{\bar{\gamma}}} \cap (\alpha \times \mu)$ is an $\alpha$-Aronszajn tree. 
	\end{enumerate}
\end{claim}

\begin{proof}
	Let $X = X_{\bar{\gamma}}$ be the set of all $\alpha < \kappa$ that satisfy these requirements for all 
	$\bar{\bar{\gamma}} \in \bar{\gamma} \cap M_\alpha$. 
	The claim follows immediately from a $\Pi^{1}_{1}$ reflection argument together with the fact that 
	$\mathcal{F}$ extends the club filter. In (7) we make use of the fact that 
	$(\mathbb{I}^2_{\bar{\bar{\gamma}}} \ast \dot{\mathbb{I}}^1_{\bar{\bar{\gamma}}}) \cap M$
	forces that $T_{\bar{\bar{\gamma}}}$ is a $\kappa$-Aronszajn tree, which follows from observing 
	that $(\mathbb{I}^2_{\bar{\bar{\gamma}}} \ast \dot{\mathbb{I}}^1_{\bar{\bar{\gamma}}}) \cap M$
	is a regular subposet of $\mathbb{I}^2_{\bar{\bar{\gamma}}} \ast \dot{\mathbb{I}}^1_{\bar{\bar{\gamma}}}$
	(this is itself a consequence of the fact that $\mathbb{I}^2_{\bar{\bar{\gamma}}} \ast \dot{\mathbb{I}}^1_{\bar{\bar{\gamma}}}$
	has the $\kappa$-c.c., by the inductive hypothesis). 
\end{proof}

\begin{defn}[\cite{laver_shelah}]
	A condition $p \in \mathbb{I}^2_\gamma$ is said to be \emph{determined} if there is in 
	$V$ a sequence $\langle x_{\xi} \colon 1 \leq \xi < \gamma \rangle$ such that for all 
	$1 \leq \xi < \gamma$, $p \restriction \xi \forces_{\mathbb{I}^2_{\xi}} p\paren{\xi} = \check{x}_{\xi}$. 
\end{defn}

As in \cite{laver_shelah}, we may easily observe that the set of determined conditions is dense in 
$\mathbb{I}^2_{\gamma}$. 

\begin{defn}[\cite{golshani_hayut}]
	Suppose that $p \in \mathbb{I}^{2}_{\gamma} \cap M$ is some condition and $\alpha < \kappa$. 
	Write $p = \langle p(\xi) \colon \xi < \gamma \rangle$. Then $p \restriction M_\alpha$ denotes the 
	condition $\langle p'(\xi) \colon \xi < \gamma \rangle$, where $p'(\xi)$ is the trivial condition if 
	$\xi \notin M_\alpha$ and $p'(\xi) = p(\xi) \cap M_\alpha$ otherwise. 
	We say that $p$ is \emph{$\alpha$-compatible} if 
	$p \restriction M_\alpha$ forces that $p$ is a determined condition in 
	$(\mathbb{I}^2_\gamma \cap M) / (G_{\mathbb{I}^2_\gamma} \cap M_\alpha)$. 
\end{defn}

\begin{claim}
	Let $X = X_{\bar{\gamma}}$ be as in the previous claim. Then for every $\alpha \in X$, 
	$\bar{\bar{\gamma}} \in (\bar{\gamma} + 1) \cap M_\alpha$, 
	$\bar{p} \in \mathbb{I}^2_\gamma \cap 
	M_\alpha$, $\alpha$-compatible $p^L, p^R \in \mathbb{I}^2_{\gamma} \cap M$ with 
	$\bar{p} = p^L \restriction M_\alpha = p^R \restriction M_\alpha$, 
	and every pair $(\dot{x}^L, \dot{x}^R)$ of 
	$(\mathbb{I}^2_{\bar{\bar{\gamma}}} \cap M) \ast (\mathbb{I}^1_{\bar{\bar{\gamma}}} \cap M_\alpha)$-names for 
	nodes in $T_{\bar{\bar{\gamma}}}$ above level $\alpha$,  
	there are 
	$\alpha$-compatible conditions $p^L_{*}, p^R_{*} \in \mathbb{I}^2_{\gamma} \cap M$, 
	$\bar{p}_* \in \mathbb{I}^2_{\gamma} \cap M_\alpha$, and a sequence 
	\begin{align*}
		\langle r_{\eta}, \xi_{\eta}, \check{x}^L_{\eta}, 
		\check{x}^R_{\eta} \colon \eta < \vartheta \rangle
	\end{align*}
	(for some $\vartheta < \mu$) in $M_\alpha$ such that: 
	\begin{enumerate}[(a)]
		\item $p^L_* \leq p^L$, $p^R_* \leq p^R$ and 
			$\bar{p}_* = p^L_* \restriction M_\alpha = p^R_* \restriction M_\alpha$. 
		\item For all $\eta < \vartheta$ $\bar{p}_* \forces_{\mathbb{I}^2_{\gamma} \cap M_\alpha} r_{\eta} \in 
			\dot{\mathbb{I}}^1_\gamma \cap M_\alpha$. 
		\item For all $\eta < \vartheta$, $\xi_{\eta} < \alpha$ and $x^L_{\eta}$, $x^R_{\eta}$ are elements of 
			$\left\{ \xi_{\eta} \right\} \times \mu$ with 
			$x^L_{\eta} \neq x^R_{\eta}$. 
		\item $(p^L_* \restriction \bar{\bar{\gamma}}, r_{\eta} \restriction \bar{\bar{\gamma}}) \forces 
			\check{x}^L_{\eta} \leq \dot{x}^L$ and 
			$(p^R_* \restriction \bar{\bar{\gamma}}, r_{\eta} \restriction \bar{\bar{\gamma}}) \forces 
			\check{x}^R_{\eta} \leq \dot{x}^R$. 
		\item $\bar{p}_{*} \forces_{\mathbb{I}^2_{\gamma} \cap M_\alpha} 
			\left\{ \dot{r}_{\eta} \colon \eta < \vartheta \right\}$ is a maximal antichain in 
			$\dot{\mathbb{I}}^1_\gamma$. 
	\end{enumerate}
	\label{separation_claim}
\end{claim}

	\begin{proof}
		Suppose $\alpha \in X$, 
		$\bar{\bar{\gamma}} \in (\bar{\gamma} + 1) \cap M_\alpha$, and 
		fix names $\dot{x}^L$, $\dot{x}^R$ for nodes in $T_{\bar{\bar{\gamma}}}$ of level 
		$\geq \alpha$ and conditions $\bar{p}$, $p^L$, $p^R$ as in the 
		statement of the claim. It follows from the choice of 
		$\alpha$ that for any $(\mathbb{I}^2_{\bar{\bar{\gamma}}} \cap M) \ast 
		(\dot{\mathbb{I}}^1_{\bar{\bar{\gamma}}} \cap M_\alpha)$-generic $G$ 
		the branches in $T_{\bar{\bar{\gamma}}} \restriction \alpha$ below 
		$\dot{x}^L$, $\dot{x}^R$ are not in 
		$V[G \cap (\mathbb{I}^2_{\bar{\bar{\gamma}}} \ast \mathbb{I}^1_{\bar{\bar{\gamma}}}) \cap M_\alpha]$. 
	
\begin{subclaim}
	For any pair $(s^L, t)$, $(s^R, t)$ of conditions in 
	$(\mathbb{I}^2_{\gamma} \cap M) \ast 
	(\dot{\mathbb{I}}^1_{\gamma} \cap M_\alpha)$ such that $s^L$, $s^R$ are $\alpha$-compatible 
	and $s^L \restriction M_\alpha = s^R \restriction M_\alpha$, 
	there is another pair $(q^L, r)$, $(q^R, r)$ of conditions in 
	$(\mathbb{I}^2_{\gamma} \cap M) \ast 
	(\dot{\mathbb{I}}^1_{\gamma} \cap M_\alpha)$ such that: 
	\begin{itemize}	
		\item $(q^L, r) \leq (s^L, t)$ and $(q^R, r) \leq (s^R, t)$
		\item $(q^L \restriction \bar{\bar{\gamma}}, r \restriction \bar{\bar{\gamma}})$, $(q^R \restriction \bar{\bar{\gamma}}, 
	r \restriction \bar{\bar{\gamma}})$ force incompatible values for the branches below 
	$\dot{x}^L$ and $\dot{x}^R$
		\item $q^L$, $q^R$ are $\alpha$-compatible 
		\item $q^L \restriction M_\alpha = q^R \restriction M_\alpha$. 
\end{itemize}
\end{subclaim}
\begin{proof}
	This is done exactly as in \cite{golshani_hayut}. We give the proof for the convenience 
	of the reader. Suppose the opposite for the sake of a contradiction, and consider 
	pairs $(s^L, t)$, $(s^R, t)$ witnessing the negation. 
	Let $H$ be $(\mathbb{I}^2_{\bar{\bar{\gamma}}} \ast \dot{\mathbb{I}}^1_{\bar{\bar{\gamma}}})
	\cap M_\alpha$-generic with $(s^L \restriction \bar{\bar{\gamma}}) \restriction M_\alpha = 
	(s^R \restriction \bar{\bar{\gamma}}) \restriction M_\alpha \in H$ and $J_i$ be 
	$(\mathbb{I}^2_{\bar{\bar{\gamma}}} \cap M) / (\mathbb{I}^2_{\bar{\bar{\gamma}}} \cap H)$-mutually 
	generic with $(s^i \restriction \bar{\bar{\gamma}}, t) \in J_i$ (for $i \in \curly{L, R}$). 
	
	If $K_i$ is any 
	$[(\mathbb{I}^2_{\bar{\bar{\gamma}}} \ast \dot{\mathbb{I}}^1_{\bar{\bar{\gamma}}}) \cap M]/ (H \ast J_i)$-
	generic ($i \in \curly{L, R}$) then in $V[H][J_i][K_i]$ there is a branch $b^i$ in the tree 
	$T_{\bar{\bar{\gamma}}} \cap \paren{\alpha \times \mu}$ consisting of nodes which lie below 
	$x^i$. Moreover, by condition (2) of Definition \ref{suitability_defn}, we have 
	$b^i \in V[H][J_i]$ (note, however, that $V[H][J_i]$ may not recognize that all nodes in 
	$b^i$ are below $x^i$, or even that $T_{\bar{\bar{\gamma}}}$ itself is a tree). 
	Nonetheless, by condition (1) of Definition \ref{suitability_defn} 
	there exists $\mu_0^i < \mu$ and a collection 
	$\{ \check{b}^i_{\xi} \colon \xi < \mu_0^i \}$ of names for elements of 
	$V[H][J_i]$ which are cofinal branches through $T_{\bar{\bar{\gamma}}} \cap 
	(\alpha \times \mu)$ such that in $V[H][J_i]$ the following holds: 
	\begin{align*}
		\forces_{ [(\mathbb{I}^2_{\bar{\bar{\gamma}}} 
		\ast \dot{\mathbb{I}}^1_{\bar{\bar{\gamma}}}) \cap M]/ (H \ast J_i)}
		\paren{\exists \xi < \mu_0^i}\paren{ \check{b}_{\xi}^i = \dot{b}^i}
	\end{align*}
	where $\dot{b}^i$ is the canonical name for the branch $b^i$ described above. 

	Moreover, by the assumption of the subclaim, there must exist 
	$\xi^{L} < \mu_0^L$ and $\xi^{R} < \mu_0^{R}$ such that 
	$b^{L}_{\xi^L} = b^{R}_{\xi^R}$. Denoting this common value by $b$, we have
	\begin{align*}
		b \in V[H][J_L] \cap V[H][J_R]
	\end{align*}
	and since $J_L$, $J_R$ were chosen to be mutually generic we have 
	$b \in V[H]$. But this is a contradiction since 
	$(\mathbb{I}^2_{\bar{\bar{\gamma}}} \ast \dot{\mathbb{I}}^1_{\bar{\bar{\gamma}}}) \cap M_\alpha$
	forces $T_{\bar{\bar{\gamma}}} \cap \paren{\alpha \times \mu}$ to be 
	an $\alpha$-Aronszajn tree. 
\end{proof}
Invoking this claim, we may find a pair of conditions $(p^L_0, r_0)$, 
$(p_0^R, r_0)$ in $ (\mathbb{I}^2_\gamma \cap M) 
\ast (\dot{\mathbb{I}}^1_\gamma \cap M_\alpha)$ with $p_0^{L} \leq p^L$, 
$p^R_0 \leq p^R$, and $p_0^L \restriction M_\alpha = p_0^R \restriction M_\alpha$ 
together with $\xi_0 < \alpha$ and elements $x_0^L$, $x_0^R$ in $\left\{ \xi_0 \right\} \times 
\mu$ such that 
\begin{align*}
	(p_0^L \restriction \bar{\bar{\gamma}}, r_0 \restriction \bar{\bar{\gamma}}) &\forces 
	\check{x}_0^L \leq \dot{x}^L \\
	(p_0^R \restriction \bar{\bar{\gamma}}, r_0 \restriction \bar{\bar{\gamma}}) &\forces 
	\check{x}_0^R \leq \dot{x}^R
\end{align*}
Furthermore, we may assume that if we let $t_0^L$ be the unique element of 
$p_0^L\!\paren{0}\!\paren{\alpha}$ and $t_0^R$ be the unique element of 
$p_0^R\!\paren{0}\!\paren{\alpha}$ then 
\begin{align*}
	\paren{p^L_0 \restriction M_\alpha}\!\paren{0} \ast \check{t}^L_0, \\
	\paren{p^R_0 \restriction M_\alpha}\!\paren{0} \ast \check{t}^R_0
\end{align*}
are flat conditions in $\mathbb{P}\!\paren{\mu, < \alpha} \ast \dot{\mathbb{T}}_{\alpha}$, 
where $\dot{\mathbb{T}}_{\alpha} = \check{\mathbb{T}}_{\bigcup{\dot{G}}_{\mathbb{P}(\mu, < \alpha)}}$. 

Proceeding inductively, suppose $\nu < \mu$ and we have defined the pairs 
$(p^L_{\eta}, \dot{r}_{\eta})$, $(p^R_{\eta}, \dot{r}_{\eta})$ in $(\mathbb{I}^2_{\gamma} \cap M) 
 \ast (\dot{\mathbb{I}}^1_\gamma \cap M_\alpha)$, 
$\bar{p}_{\eta}$ in $\mathbb{I}^2_{\gamma} \cap M_\alpha$, and $\check{t}^{L}_{\eta}$, 
$\check{t}^{R}_{\eta}$ in $\dot{\mathbb{T}}_{\alpha} \cap M$ together with 
$\xi_{\eta}$ and $x^L_{\eta}$, $x^{R}_{\eta} \in \curly{\xi_{\eta}} \times \mu$, 
such that: 
\begin{itemize}
	\item The sequences $\langle p^{L}_{\eta} \colon \eta < \nu \rangle$ and 
		$\langle p^{R}_{\eta} \colon \eta < \nu \rangle$ are decreasing and 
		for each $\eta$ $p^L_{\eta}$ and $p^{R}_{\eta}$ are $\alpha$-compatible. 
	\item $\bar{p}_{\eta} = p^L_{\eta} \restriction M_\alpha = p^R_{\eta} \restriction M_\alpha$. 
	\item $\bar{p}_{\eta} \forces_{\mathbb{I}^2_{\gamma} \cap M_\alpha} \dot{r}_{\eta} \in 
		\dot{\mathbb{I}}^1_{\gamma} \cap M_\alpha$. 
	\item For $\eta_0 < \eta_1 < \nu$, $\bar{p}_{\eta_1} \forces_{\mathbb{I}^2_{\delta} \cap M_\alpha}$
		$\dot{r}_{\eta_0}$, $\dot{r}_{\eta_1}$ are incompatible. 
	\item $\xi_n < \alpha$, $x^{L}_{\eta}$, $x^{R}_{\eta} \in \curly{\xi_\eta} \times \mu$
		and $x^{L}_{\eta} \neq x^{R}_{\eta}$. 
	\item $(p^{L}_{\eta} \restriction \bar{\bar{\gamma}}, \dot{r}_{\eta} \restriction \bar{\bar{\gamma}}) \forces 
		\check{x}^{L}_{\eta} \leq \dot{x}^L$. 
	\item $(p^{R}_{\eta} \restriction \bar{\bar{\gamma}}, \dot{r}_{\eta} \restriction \bar{\bar{\gamma}}) \forces 
		\check{x}^{R}_{\eta} \leq \dot{x}^R$. 
	\item $t^{L}_{\eta}$ is the unique element of $p^{L}_{\eta}\!\paren{0}\!\paren{\alpha}$. 
	\item $t^{R}_{\eta}$ is the unique element of $p^{R}_{\eta}\!\paren{0}\!\paren{\alpha}$. 
	\item $(p^{L}_{\eta} \restriction M_\alpha)\!\paren{0} \ast 
		\check{t}^{L}_{\eta}$, $(p^{R}_{\eta} \restriction M_\alpha)\!\paren{0} 
		\ast \check{t}^{R}_{\eta}$ are flat conditions in 
		$\mathbb{P}(\mu, < \alpha) \ast \dot{\mathbb{T}}_{\alpha}$, where
		$\dot{\mathbb{T}}_{\alpha} = \check{\mathbb{T}}_{\bigcup \dot{G}_{\mathbb{P}(\mu, < \alpha)}}$. 
\end{itemize}
If $\nu$ is a successor ordinal let $q_{\nu}^{L} = p^{L}_{\nu - 1}$, $q^{R}_{\nu} = p^{R}_{\nu - 1}$. 
Otherwise, let 
\begin{align*}
	t^{L}_{\nu} &= \bigcup_{\eta < \nu}{t^{L}_{\eta}} \cup \curly{ \sup \bigcup_{\eta < \nu}{t^{L}_{\eta}}}\\
	t^{R}_{\nu} &= \bigcup_{\eta < \nu}{t^{R}_{\eta}} \cup \curly{ \sup \bigcup_{\eta < \nu}{t^{R}_{\eta}}}
\end{align*}
and let $q^{L}_{\nu}$, $q^{R}_{\nu}$ be lower bounds of 
$\left\{ p^{L}_{\eta} \colon \eta < \nu \right\}$, $\left\{ p^{R}_{\eta} \colon \eta < \nu \right\}$
such that \emph{both} $t^{L}_{\nu}$, $t^{R}_{\nu}$ appear on the (approximations to) square sequences 
$q^{L}_{\nu}\!\paren{0}$, $q^{R}_{\nu}\!\paren{0}$. These lower bounds may be seen to exist by an argument 
similar to that used to prove strategic closure in Lemma \ref{basic_properties_of_p_lem}. Namely, each initial 
segment of $t^{L}_{\eta}$, $t^{R}_{\eta}$ of limit order type has already been placed on $p^{L}_{\eta}(0)$, 
$p^{R}_{\eta}(0)$ for some $\eta < \nu$, and therefore we may place $t^{L}_{\nu}$, $t^{R}_{\nu}$ on 
$q^{L}_{\nu}(0)$, $q^{R}_{\nu}(0)$ without any danger of violating coherence. 
Observe that this is the part of the argument where we exploit the ``two-ness'' of the principle 
$\square_{\kappa, 2}$ (and hence of the poset used to force it). Namely, we seek to ensure that 
$p^{L}_{*}$, $p^{R}_{*}$ agree on $M_\alpha$, and so must put both threads on both conditions. 
Finally, note that in either case ($\nu$ successor or limit) we have 
$q_{\nu}^{L}, q_{\nu}^{R}$ are $\alpha$-compatible conditions in  $\mathbb{I}^{2}_{\gamma} \cap M$ 
and $q^{L}_{\nu} \restriction M_\alpha = q^{R}_{\nu} \restriction M_\alpha$. 
Let $\bar{q}_{\nu} = q^L_{\nu} \restriction M_\alpha = q^R_{\nu} \restriction M_\alpha$. 
If 
\begin{align*}
	\bar{q}_{\nu} \forces_{\mathbb{I}^{2}_{\gamma} \cap M_\alpha} \curly{\dot{r}_\eta \colon \eta < \nu}
	\text{ is a maximal antichain}
\end{align*}
we halt the construction and set $p^{L}_{\nu} = q_{\nu}^{L}$, $p^{R}_{\nu} = q^{R}_{\nu}$. 
Otherwise proceed exactly as when obtaining $r_{0}$, except now working below $s_{\nu}$. Namely, 
find a condition $s_{\nu}$ forced to be incompatible with 
every $r_{\eta}$ ($\eta < \nu$) and choose $(p^L_{\nu}, r_{\nu})$, $(p^{R}_{\nu}, r_\nu)$, 
$\bar{p}_{\nu}$, $\xi_\nu < \alpha$, and $x^L_{\nu}, x^{R}_{\nu} \in \curly{\xi_\nu}
\times \mu$ such that: 
\begin{itemize}
	\item $(p^L_\nu, r_\nu), (p^R_\nu, r_\nu) \in (\mathbb{I}^2_\gamma \cap M) 
		\ast ( \dot{\mathbb{I}}^1_\gamma \cap M_\alpha)$. 
	\item $p^{L}_{\nu}$, $p^R_{\nu}$ are $\alpha$-compatible. 
	\item $(p^L_\nu, r_\nu) \leq (q^L_\nu, s_\nu)$ and $(p^R_\nu, r_\nu) \leq (q^R_\nu, s_\nu)$. 
	\item $\bar{p}_\nu = p^L_\nu \restriction M_\alpha = p^R_\nu \restriction M_\alpha$. 
	\item $(p^L_\nu \restriction \bar{\bar{\gamma}}, \dot{r}_\nu \restriction \bar{\bar{\gamma}}) \forces 
		\check{x}^L_\nu \leq \dot{x}^L$. 
	\item $(p^R_{\nu} \restriction \bar{\bar{\gamma}}, \dot{r}_{\nu} \restriction \bar{\bar{\gamma}}) \forces 
		\check{x}^R_\nu \leq \dot{x}^R$. 
	\item $t_{\nu}^L$ is the unique element of $p_{\nu}^L(0)(\alpha)$. 
	\item $t_{\nu}^R$ is the unique element of $p_{\nu}^R(0)(\alpha)$. 
	\item $(p_{\nu}^L \restriction M_\alpha)(0) \ast \check{t}^L_\nu$, 
		$(p_{\nu}^{R} \restriction M_\alpha)(0) \ast \check{t}^R_\nu$ are flat conditions 
		in $\mathbb{P}(\mu, < \alpha) \ast \dot{\mathbb{T}}_\alpha$, where 
		$\dot{\mathbb{T}}_\alpha = \check{\mathbb{T}}_{\bigcup \dot{G}_{\mathbb{P}(\mu, < \alpha)}}$
\end{itemize}
By Lemma \ref{chain_condition_of_I1_lem} this process terminates after 
$< \mu$ many steps. At its completion we get an ordinal 
$\vartheta < \mu$, descending sequences $\langle p^{L}_{\eta} \colon \eta \leq \vartheta \rangle$
and $\langle p^{R}_{\eta} \colon \eta \leq \vartheta \rangle$ of conditions in 
$\mathbb{I}^2_{\gamma} \cap M$, as well 
as sequences $\langle \bar{p}_{\eta} \colon \eta < \vartheta \rangle$, 
$\langle r_{\eta} \colon \eta < \vartheta \rangle$, and $\langle (\xi_{\eta}, \check{t}^L_{\eta}, \check{t}^R_{\eta}, 
\check{x}^L_{\eta}, \check{x}^R_{\eta}) \colon \eta < \vartheta \rangle$ such that: 

\begin{itemize}
	\item $(p^L_\eta, r_\eta), (p^R_\eta, r_\eta) \in (\mathbb{I}^2_\gamma \cap M) 
		\ast ( \dot{\mathbb{I}}^1_\gamma \cap M_\alpha)$. 
	\item $p^{L}_{\eta}$, $p^{R}_{\eta}$ are $\alpha$-compatible. 
	\item $\bar{p}_\eta = p^L_\eta \restriction M_\alpha = p^R_\nu \restriction M_\alpha$. 
	\item $x^{L}_{\eta}, x^{R}_{\eta} \in \curly{\xi_{\eta}} \times \mu$. 
	\item $(p^L_\eta \restriction \bar{\bar{\gamma}}, \dot{r}_\eta \restriction \bar{\bar{\gamma}}) \forces 
		\check{x}^L_\eta \leq \dot{x}^L$. 
	\item $(p^R_{\eta} \restriction \bar{\bar{\gamma}}, \dot{r}_{\eta} \restriction \bar{\bar{\gamma}}) \forces 
		\check{x}^R_{\eta} \leq \dot{x}^R$. 
	\item $t_{\eta}^L$ is the unique element of $p_{\eta}^L(0)(\alpha)$. 
	\item $t_{\eta}^R$ is the unique element of $p_{\eta}^R(0)(\alpha)$. 
	\item $(p_{\eta}^L \restriction M_\alpha)(0) \ast \check{t}^L_\eta$, 
		$(p_{\eta}^{R} \restriction M_\alpha)(0) \ast \check{t}^R_\eta$ are flat conditions 
		in $\mathbb{P}(\mu, < \alpha) \ast \dot{\mathbb{T}}_\alpha$, where 
		$\dot{\mathbb{T}}_\alpha = \check{\mathbb{T}}_{\bigcup \dot{G}_{\mathbb{P}(\mu, < \alpha)}}$
\end{itemize}
Finally, set $p^L_{*} = p^L_{\vartheta}$, $p^R_{*} = p^R_{\vartheta}$. 
Then this 
pair $(p^L_{*}, p^R_{*})$ together with $\langle r_{\eta}, \xi_{\eta}, \check{x}^L_{\eta}, \check{x}^R_{\eta} 
\colon \eta < \vartheta \rangle$ are as desired. 
	\end{proof} 
	Following \cite{golshani_hayut}, let
	us call the sequence $\langle (r_{\eta}, \xi_{\eta}, \check{x}^L_{\eta}, \check{x}^R_{\eta}) 
	\colon \eta < \vartheta \rangle$
	an $\alpha$-separating witness for the nodes $\dot{x}^{L}$, $\dot{x}^R$ relative to $p_{*}^L$, $p_{*}^R$. 
	We now continue with the proof of Lemma \ref{chain_condition_lem}: 
	\begin{claim}
		There is $X \in \mathcal{F}$ such that for every condition 
		$p \in \mathbb{I}^2_\gamma \cap M$ and $\alpha \in X$ there are conditions 
		$p^L, p^R \leq p$ (both in $M$) such that $p^L \restriction M_\alpha = p^R \restriction 
		M_\alpha$ and for every $\gamma' \in \gamma \cap M_\alpha$, any pair of elements above 
		$\alpha$ in $\dom{(p^L(\gamma'))} \times \dom{(p^R(\gamma'))}$ has an 
		$\alpha$-separating witness in $M_\alpha$ relative to $p^L \restriction \gamma$, 
		$p^R \restriction \gamma$. 
		\label{separating_pair_claim}
	\end{claim}
	We call such a pair $(p^L, p^R)$ an \emph{$\alpha$-separating pair}. 
	\begin{proof}
		Recall that we chose a sequence $\left\{ \gamma_i \colon i < \cf{\gamma} \right\}$ 
		cofinal in $\gamma$. For each $i < \cf{\gamma}$ let $X_{\gamma_i}$ be as in 
		Claim \ref{separation_claim}, and let $X = \bigcap_{i < \cf{\gamma}}{X_{\gamma_i}}$. This 
		$X$ suffices, as may be seen by applying Claim \ref{separation_claim} $\cf{\gamma}$ many times 
		and using the $< \mu$-strategic closure of $\mathbb{I}^2_{\gamma}$. 
	\end{proof}
	Returning to the proof of the $\kappa$-c.c., let $X$ be as in Claim \ref{separating_pair_claim} and 
	let $\langle p_\alpha \colon \alpha < \kappa \rangle \in M$ be a sequence of 
	conditions in $\mathbb{I}^2_\gamma$. For every $\alpha \in X$ we may extend 
	$p_\alpha$ to an $\alpha$-separating pair $(p_\alpha^L, p_\alpha^R) \in M$. 
	Let $s_\alpha \in M_\alpha$ be the list of separating witnesses and let 
	$\bar{p}_\alpha$ denote $p_\alpha^L \restriction M_\alpha = p_\alpha^R \restriction M_\alpha$. 

	The function $\alpha \mapsto (s_\alpha, \bar{p}_\alpha)$ is regressive, and so by normality of
	$\mathcal{F}$ there is a set $Y$ which is positive with respect to 
	this filter and a pair $(s^*, \bar{p}^*)$ such that for all $\alpha \in Y$ 
	$(s_\alpha, \bar{p}_\alpha) = (s^*, \bar{p}^*)$. By further thinning we may assume that for 
	every $\alpha_0, \alpha_1 \in Y$ with $\alpha_0 < \alpha_1$ we have 
	$p_{\alpha_0}^L, p_{\alpha_1}^R \in M_{\alpha_1}$. Similarly, we may assume without loss of 
	generality that 
	\begin{align*}
		\left\{ \supp{(p_\alpha^L)} \cup \supp{(p_\alpha^R)} \colon \alpha \in Y \right\}
	\end{align*}
	is a $\Delta$-system with root $R$. 

	We claim that for any $\alpha_0 < \alpha_1$ in $Y$, $p_{\alpha_0}$ is compatible with 
	$p_{\alpha_1}$, as witnessed by the condition $q$ given by 
	$q(\gamma') = p_{\alpha_0}^{L}(\gamma') \cup p_{\alpha_1}^{R}(\gamma')$ for every 
	$\gamma' < \gamma$. 

	We must show that $q$ so defined is a condition. Clearly $|\dom{(q)}| < \mu$ and 
	so it remains only to show that $q \restriction \gamma'$ forces that 
	$q(\gamma')$ is a condition in $\dot{\mathbb{J}}^{2}_{\gamma'}$ for all 
	$\gamma' < \gamma$. We proceed by induction on $\gamma'$. For 
	$\gamma' = 0$ $q(\gamma') \in \mathbb{P}(\mu, < \kappa)$, 
	since $p^L_{\alpha_0}(\gamma')$, $p^R_{\alpha_1}(\gamma')$ have identical intersection with 
	$M_{\alpha_0}$ and are disjoint above $\alpha_0$. 

	Now assume $\gamma' > 0$ and $q \restriction \gamma'$ is a condition. 
	Without loss of generality $\dot{T}_{\gamma'_0}$
	is an $\mathbb{I}^2_{\gamma'} \ast \dot{\mathbb{I}}^1_{\gamma'}$-name for a 
	$\kappa$-Aronszajn tree, as otherwise $\dot{\mathbb{J}}^2_{\gamma'}$ is a name for 
	the trivial forcing. We may also assume $\gamma' \in R$, since otherwise 
	$q(\gamma')$ is either $p^L_{\alpha_0}(\gamma')$ or $p_{\alpha_1}^R(\gamma')$. 

	In order to show that $q \restriction \gamma' \forces q(\gamma')$ is a condition
	we must show that if $\dot{x}^L, \dot{x}^R \in \dom{(q(\gamma'))}$ and 
	$q \restriction \gamma' \forces q(\gamma')(\dot{x}^L) = q(\gamma')(\dot{x}^R)$, then 
	\begin{align*}
		(q \restriction \gamma', \dot{1}_{\mathbb{I}^1_{\gamma'}}) \forces_{\mathbb{I}^2_{\gamma'} 
		\ast \dot{\mathbb{I}}^1_{\gamma'}} \dot{x}^L \perp \dot{x}^R
	\end{align*}
	Since $p_{\alpha_0}^L \restriction M_{\alpha_0} = \bar{p} = p^R_{\alpha_1} \restriction M_{\alpha_1}$, 
	we may assume without loss of generality that both $\dot{x}^L$, $\dot{x}^R$ are names for nodes 
	above level $\alpha_0$. Letting $s^* = \langle r_{\eta}, \xi_{\eta}, \check{x}^L_{\eta}, 
	\check{x}^R_{\eta} \colon \eta < \vartheta \rangle$ be our fixed separating witness, 
	we have: 
	\begin{align*}
		(p_{\alpha_0}^L \restriction \gamma', \dot{r}_\eta \restriction \gamma') &\forces 
		\check{x}^L_{\eta} \leq \dot{x}^L \\
		(p_{\alpha_1}^R \restriction \gamma', \dot{r}_\eta \restriction \gamma') &\forces 
		\check{x}^R_{\eta} \leq \dot{x}^R
	\end{align*}
	By the induction hypothesis $q \restriction \gamma'$ is a condition extending 
	both $p_{\alpha_0}^L \restriction \gamma'$ and $p_{\alpha_0}^R \restriction \gamma'$
	and so in particular, since $\check{x}^{L}_{\eta} \neq \check{x}^{R}_{\eta}$, for all $\eta < \vartheta$ we have 
	\begin{align*}
		(q \restriction \gamma', \dot{r}_{\eta} \restriction \gamma') \forces 
		\dot{x}^L \perp_{\dot{T}_{\gamma'}} \dot{x}^R
	\end{align*}
	Since $\left\{ \dot{r}_\eta \colon \eta < \vartheta \right\}$ is forced to be a maximal antichain, we have 
	\begin{align*}
		(q \restriction \gamma', \dot{1}_{\mathbb{I}^2_{\gamma'}}) \forces \dot{x}^L 
		\perp_{\dot{T}_{\gamma'}} \dot{x}^R
	\end{align*} 
	as desired. 
\end{proof} 

\begin{remark}
	It behooves us to observe that the proof of Lemma \ref{chain_condition_lem} actually gives us something 
	stronger--namely that for any $S \subseteq \delta$ such that 
	$\mathbb{I}^2_{\delta} \restriction S$ is a regular subposet of $\mathbb{I}^2_{\delta}$, and 
	for any $\mathbb{I}^2_{\delta} \restriction S$-generic $K$, the quotient 
	$\mathbb{I}^2_{\delta} / K$ is $\kappa$-Knaster. The proof of this is almost 
	identical to that of Lemma \ref{chain_condition_lem}, but we must observe that 
	every name $\dot{T}_{\gamma}$ for a $\kappa$-Aronszajn tree considered in the 
	iteration remains a name for a $\kappa$-Aronszajn tree in $V[K]$.
	This is true since any such tree is $\kappa$-Aronszajn (in fact, special)
	in $V[K][L]$, where $L$ is generic for $\mathbb{I}^2_{\delta} / K$. 
	\label{stronger_lem_remark}
\end{remark}

From Lemmas \ref{chain_condition_of_I1_lem} and \ref{chain_condition_lem} we have 

\begin{lem}
	For all $\gamma \leq \delta$ $\mathbb{I}^2_{\gamma} \ast \dot{\mathbb{I}}^1_\gamma$ 
	has the $\kappa$-c.c. 
	\label{second_chain_condition_lem}
\end{lem}

\begin{lem}
	Let $\mathbb{I} = \mathbb{I}^2_{\delta} \ast \dot{\mathbb{I}}^1_\delta$ be as 
	above and suppose that $G$ is $\mathbb{I}$-generic over $V$. Then: 
	\begin{enumerate}[(a)]
		\item $\mu$ remains a cardinal in $V[G]$, $(\mu^+)^{V[G]} = \kappa$, and 
			$(\mu^{++})^{V[G]} = (\kappa^+)^{V}$. 
		\item $V[G] \models 2^{\mu} \geq \delta$. 
		\item $V[G] \models \square_{\mu, 2}$. 
	\end{enumerate}
\end{lem}

Now we are ready to complete the proof of Theorem \ref{main_theorem}. 

\begin{lem}
	Suppose that $X \in V[G]$ and $X \subseteq \kappa$. Then 
	$X \in V[G_{\mathbb{I}^2_{\gamma} \ast \dot{\mathbb{I}}^1_\gamma}]$
	for some $\gamma < \delta$. 
\end{lem}

\begin{proof}
	Immediate from Lemma \ref{second_chain_condition_lem}. 
\end{proof}

\begin{lem}
	The poset $\mathbb{I}$ forces $\mathsf{SATP}(\kappa)$. 
\end{lem}

\begin{proof}
	Suppose that $T$ is an $\kappa$-Aronszajn tree in 
	$V[G_{\mathbb{I}}]$ and let $\dot{T}$ be a 
	canonical name for it. Then for some 
	$\gamma < \delta$ it is an $\mathbb{I}_\gamma$-name 
	and $\dot{T} = \dot{T}_{\gamma}$. By construction of the 
	forcing poset, $\forces_{\mathbb{I}_{\gamma + 1}} ``\dot{T}$ is special''
	and since $V[G_{\mathbb{I}_{\gamma + 1}}]$ and $V[G_{\mathbb{I}}]$ have the same 
	cardinals $T$ remains special in the latter generic extension. 
\end{proof}

This concludes the proof of Theorem \ref{main_theorem}.

\begin{cor}
	Suppose that there is a weakly compact cardinal. Then in some generic 
	extension of $V$ the following holds: 
	\begin{align*}
		\square_{\omega_1, 2} + \mathsf{SATP}(\aleph_2) + \mathsf{SATP}(\aleph_1)
	\end{align*}
\end{cor}

\begin{proof}
	Let $\kappa$ be weakly compact and let 
	\begin{align*}
		\vec{\mathbb{I}}^2 &= \langle \langle \mathbb{I}^2_{\gamma} \colon \gamma 
		\leq \delta \rangle, \langle \dot{\mathbb{J}}^2_{\gamma} \colon \gamma < \delta \rangle \rangle \\
		\dot{\vec{\mathbb{I}}}^1 &= 
		\langle \langle \dot{\mathbb{I}}^1_{\gamma} \colon \gamma \leq \delta \rangle, 
		\langle \dot{\mathbb{J}}^1_{\gamma} \colon \gamma < \delta \rangle \rangle 
	\end{align*}
	be two iterations such that $\vec{\mathbb{I}}^2$ collapses $\kappa$ to $\aleph_2$, 
	adds $\square_{\omega_1, 2}$, and specializes all $\kappa$-Aronszajn trees while 
	anticipating $\dot{\vec{\mathbb{I}}}^1$--where $\dot{\vec{\mathbb{I}}}^1$ is an 
	$\vec{\mathbb{I}}^2$-name for the Baumgartner forcing--i.e. a finite support iteration of 
	posets of finite approximations to specializing functions of trees chosen according to 
	some appropriate bookkeeping function. For each ordinal $\gamma \leq \delta$ let $M$ 
	be an arbitrary elementary substructure of 
	$H(\theta)$ ($\theta$ sufficiently large) of cardinality $\kappa$ such that 
	$V_\kappa \cup M^{< \kappa} \cup \curly{\gamma} \subseteq M$ and $M$ contains all relevant parameters, and 
	let $\phi \colon \kappa \rightarrow M$ be an arbitrary bijection. Then 
	$\vec{\mathbb{I}}^2$, $\dot{\vec{\mathbb{I}}}^1$ are suitable for $M$, 
	$\phi$, $\gamma$ by the remark after the proof of Lemma \ref{chain_condition_lem}
	together with \cite{baumgartner_malitz_reinhardt}. Therefore the result follows immediately from 
	Theorem \ref{main_theorem}. 
\end{proof}

\section{The Global Result}

Using the methods of \cite{golshani_hayut} we are able to obtain our 
result simulaneously at $\omega$ many successive cardinals. 

\begin{theorem}
	Let $\mu$ be an uncountable regular cardinal and assume that there are $\omega$ many supercompact 
	cardinals above $\mu$. Then there is a generic extension of $V$ in which 
	$\square_{\mu^{+n}, 2} + \mathsf{SATP}(\mu^{+n + 1})$ holds for all $n \geq 0$. 
	Furthermore, if $\mu = \aleph_1$, we may ensure that $\mathsf{SATP}(\aleph_1)$ holds 
	as well. 
	\label{countably_many_theorem}
\end{theorem}

\begin{proof}
	Let $\langle \kappa_n \colon n < \omega \rangle$ 
	be an increasing sequence of indestructibly supercompact cardinals above $\mu$ and 
	let $\delta = ( \sup_{n < \omega}{\kappa_n})^{++}$. 
	In the following it will be convenient to write $\kappa_{-1} = \mu$. 
	If $\mu = \aleph_1$, let $R = \omega \cup \curly{-1}$, and otherwise 
	let $R = \omega$. 
	Let $h \colon \delta \setminus \curly{0} \rightarrow R$ be a function 
	such that for all $n \in R$, $h(\alpha) = n$ for unboundedly 
	many $\alpha < \delta$. We define an iteration 
	\begin{align*}
		\vec{\mathbb{I}} = 
		\langle \langle \mathbb{I}_{\gamma} \colon 
		\gamma \leq \delta \rangle, 
		\langle 
		\dot{\mathbb{J}}_{\gamma} \colon \gamma < \delta 
		\rangle \rangle 
	\end{align*}
	as well as auxilliary forcings 
	$\mathbb{I}_{\gamma}(> \kappa_n)$, $\dot{\mathbb{I}}_{\gamma}
	(\kappa_n)$, and $\dot{\mathbb{I}}_\gamma(< \kappa_n)$
	for $n \in R$, $\gamma \leq \delta$ such that: 
	\begin{enumerate}[(a)]
		\item $\mathbb{I}_{\gamma} \simeq 
			\mathbb{I}_{\gamma}(> \kappa_n) \ast 
			\dot{\mathbb{I}}_\gamma(\kappa_n) 
			\ast \dot{\mathbb{I}}_\gamma( < \kappa_n)$ 
		\item $\mathbb{I}_{\gamma}( > \kappa_n)$ is 
			$< \kappa_n$-strategically closed. 
		\item For all $n \geq 0$, $\forces_{\mathbb{I}_\gamma(> \kappa_n)} 
			\dot{\mathbb{I}}_\gamma(\kappa_n)$ 
			is $\kappa_n$-c.c. and $< \kappa_{n - 1}$-
			strategically closed, and for 
			$n = - 1$, $\forces_{\mathbb{I}_\gamma(> \kappa_n)} \dot{\mathbb{I}}_{\gamma}(\kappa_n)$ 
			is c.c.c. 
		\item For all $n \geq 0$, $\forces_{\mathbb{I}_\gamma(> \kappa_n) 
			\ast \dot{\mathbb{I}}_\gamma(\kappa_n) }
			\dot{\mathbb{I}}_\gamma( < \kappa_n)$
			is $\kappa_{n - 1}$-c.c. and 
			$< \mu$-closed. 
	\end{enumerate}
	Set $\mathbb{I}_1 = \prod_{n < \omega}{
		\mathbb{P}(\kappa_{n - 1}, < \kappa_n)}$, 
		where the product is taken with full support, and
		then let 
\begin{itemize}
	\item $\mathbb{I}_1(< \kappa_n) = 
		\prod_{m < n}{\mathbb{P}(\kappa_{m - 1}, < \kappa_m)}$ for $n \geq 0$ and 
		is the trivial forcing for $n = -1$. 
	\item $\mathbb{I}_1(\kappa_n) = \mathbb{P}(\kappa_{n - 1}, 
		< \kappa_n)$ for $n \geq 0$ and is the trivial forcing for $n = -1$. 
	\item $\mathbb{I}_1(> \kappa_n) = \prod_{m > n}{
			\mathbb{P}(\kappa_{m - 1}, < \kappa_m)}$, 
			also taken with full support. 
\end{itemize}

Now suppose that $2 \leq \gamma \leq \delta$ and we have already 
defined $\mathbb{I}_{\gamma'}$, $\mathbb{I}_{\gamma'}(> \kappa_n)$, 
$\dot{\mathbb{I}}_{\gamma'}(\kappa_n)$, and 
$\dot{\mathbb{I}}_{\gamma'}(< \kappa_n)$ for all $\gamma' < \gamma$
and $n \in R$. We define $\mathbb{I}_{\gamma}$, 
$\dot{\mathbb{I}}_\gamma(> \kappa_n)$, and 
$\dot{\mathbb{I}}_\gamma(< \kappa_n)$ as follows: 

\begin{itemize}
	\item If $\gamma$ is a limit ordinal then 
		$\mathbb{I}_\gamma$ is the set of all 
		$p$ with domain $\gamma$ such that 
		$p \restriction \gamma' \in \mathbb{I}_{\gamma'}$ 
		for all $\gamma' < \gamma$,  for all 
		$n \geq 0$ we have $| \supp{(p)} \cap h^{-1}(n) | < 
		\kappa_{n - 1}$, and for $n = -1$ we have $| \supp{(p)} \cap h^{-1}(n) |$ finite. 
	\item If $\gamma = \bar{\gamma} + 1$ is a successor ordinal 
		and $n = h(\bar{\gamma})$ then let 
		$\dot{T}_{\bar{\gamma}}$ be an 
		$\mathbb{I}_{\bar{\gamma}}$-name for a 
		$\kappa_n$-Aronszajn tree 
		chosen according to 
		some bookkeeping function, and let
		\begin{align*}
			\mathbb{I}_{\gamma} = 
			\mathbb{I}_{\bar{\gamma}} \ast 
			\dot{\mathbb{B}}_{\mathbb{I}_{\bar{\gamma}}
		(< \kappa_n)}(\dot{T}_{\bar{\gamma}})
		\end{align*}
		Observe that $\dot{\mathbb{B}}_{\mathbb{I}_{\bar{\gamma}}(< \kappa_n)}(\dot{T}_{\bar{\gamma}})$
		is an $\mathbb{I}_{\bar{\gamma}}(> \kappa_{n - 1})$-name (if $n = - 1$ we mean here that it 
		is an $\mathbb{I}_{\bar{\gamma}}$-name) and so may be viewed as 
		an $\mathbb{I}_{\bar{\gamma}}(\kappa_n)$-name in the extension by 
		$\mathbb{I}_{\bar{\gamma}}(> \kappa_n)$. 
\end{itemize}

For $n \in R$ let 
\begin{align*}
	\mathbb{I}_{\gamma}(> \kappa_n) = 
	\left\{  p \in \mathbb{I}_\gamma \colon p(0) \in \mathbb{I}_1(> \kappa_n) \land 
	\supp{(p)} \setminus \curly{0} \subseteq \bigcup_{m > n}{h^{-1}(m)}\right\}
\end{align*}
Then $\mathbb{I}_\gamma(> \kappa_n)$ is a regular subforcing of 
$\mathbb{I}_\gamma$. We let $\dot{\mathbb{I}}_\gamma(\kappa_n)$ be an 
$\mathbb{I}_\gamma(> \kappa_n)$-name for the poset 
\begin{align*}
	\mathbb{I}_{\gamma}(\kappa_n) = 
	\left\{ p \in \mathbb{I}_\gamma \colon p(0) \in \mathbb{I}_1(\kappa_n) \land 
	\supp{(p)} \setminus \curly{0} \subseteq h^{-1}(n) \right\}
\end{align*}
and let $\dot{\mathbb{I}}_\gamma(< \kappa_n)$ be an $\mathbb{I}_\gamma(> \kappa_n) 
\ast \dot{\mathbb{I}}_\gamma(\kappa_n)$-name for the poset 
\begin{align*}
	\mathbb{I}_\gamma(< \kappa_n) = 
	\left\{ p \in \mathbb{I}_\gamma \colon p(0) \in \mathbb{I}_1(< \kappa_n) \land
	\supp{(p)} \setminus \curly{0} \subseteq \bigcup_{m < n}{h^{-1}(m)} \right\} 
\end{align*}

Observe that $\mathbb{I}_\gamma \simeq \mathbb{I}_{\gamma}(> \kappa_n) \ast 
\dot{\mathbb{I}}_\gamma(\kappa_n) \ast \dot{\mathbb{I}}_\gamma(< \kappa_n)$. 

\begin{lem}
Let $G_{> \kappa_n}$ be generic for $\mathbb{I}_{\gamma}(> \kappa_n)$ and 
	$\dot{G}_{\kappa_n}$ be an $\mathbb{I}_\gamma(> \kappa_n)$-name for the 
	generic for $\dot{\mathbb{I}}_{\gamma}(\kappa_n)$. If 
	$n \geq 0$, then $V[G_{> \kappa_n}] \models$ ``$\mathbb{I}_{\gamma}(\kappa_n)$ is 
	$< \kappa_{n - 1}$-strategically closed'' and $V[G_{> \kappa_n} \ast 
	\dot{G}_{\kappa_n}] \models$ ``$\dot{\mathbb{I}}_\gamma(< \kappa_n)$ is 
	$< \mu$-strategically closed.''
	\label{countably_many_closure_lem}
\end{lem}

\begin{proof}
	Clearly 
	\begin{align*}
	V\brac{G_{> \kappa_n}} \models 
	\text{$\mathbb{I}_\gamma(\kappa_n)$ is $< \kappa_{n - 1}$-strategically closed}
\end{align*}
since $\mathbb{I}_{\gamma}(\kappa_n)$ may be defined as an iteration with $< \kappa_{n - 1}$-support 
in $V[G_{> \kappa_n}]$ where each iterand has the requisite closure. The fact that 
\begin{align*}
	V[G_{> \kappa_n} \ast \dot{G}_{\kappa_n}] \models \dot{\mathbb{I}}_{\gamma}(< \kappa_n) \, 
	\text{ is $< \mu$-strategically closed. }
\end{align*}
may be argued similarly. 
\end{proof}

\begin{lem}
	Let $G_{> \kappa_n}$, $\dot{G}_{\kappa_n}$ be as in the statement of Lemma \ref{countably_many_closure_lem}. 
	Then $V\brac{G_{> \kappa_n}} \models$ ``$\mathbb{I}_\gamma(\kappa_n)$ is 
	$\kappa_n$-Knaster'' and $V[G_{> \kappa_n} \ast \dot{G}_{\kappa_n}] \models$ ``$\dot{\mathbb{I}}_\gamma(< 
	\kappa_n)$ is $\kappa_{n - 1}$-Knaster.'' Moreover, we actually have 
	$V[G_{> \kappa_n}] \models$ ``$\mathbb{I}_\gamma(\kappa_n)/ L$ is $\kappa_n$-Knaster,''
	for any $L$ which is generic for a regular subiteration of $\mathbb{I}_\gamma(\kappa_n)$. 
	\label{countably_many_chain_condition_lem}
\end{lem}

\begin{proof}
We prove these simultaneously using induction on $n$. 
For each $m > n$ let $G_{\kappa_m}$ denote the generic for 
$\mathbb{I}_\gamma(\kappa_m)$ and let $\dot{\mathbb{T}}(\kappa_m) = \check{\mathbb{T}}_{\bigcup \dot{G}_{\kappa_m}}$
be the $\mathbb{I}_\gamma(\kappa_m)$-name for the poset which threads $\bigcup G_{\kappa_m}$ with conditions of 
size $< \kappa_{m - 1}$. An argument similar to that given in Lemma \ref{square_star_threading_lem} tells us that 
$\mathbb{I}_\gamma(\kappa_m) \ast \dot{\mathbb{T}}(\kappa_m)$ is $< \kappa_{m - 1}$-closed. Moreover, it
is clear that this poset forces $\lvert \kappa_m \rvert = \kappa_{m - 1}$ and has size $\leq \delta$. 
Let $\dot{\mathbb{T}}(> \kappa_n)$ be the $\mathbb{I}_\gamma(> \kappa_n)$-name for 
\begin{align*}
	\mathbb{T}(> \kappa_n) = \prod_{m > n}{ \mathbb{T}(\kappa_m)}
\end{align*}
where $\mathbb{T}(\kappa_m) = \dot{\mathbb{T}}(\kappa_m)\brac{G_{\kappa_m}}$ and the product is taken with 
full support. Then $\mathbb{I}_{\gamma}(> \kappa_n) \ast \dot{\mathbb{T}}(> \kappa_n)$ is 
$< \kappa_n$-closed, forces $\lvert \kappa_m \rvert = \kappa_n$ for all $m > n$, and has 
cardinality $\leq \delta$, and so there is a regular embedding from 
$\mathbb{I}_\gamma(> \kappa_n) \ast \dot{\mathbb{T}}(> \kappa_n)$ into 
$\Col{(\kappa_n, \delta)}$--in fact we have $\Col{(\kappa_n, \delta)} \simeq (\mathbb{I}_\gamma(> \kappa_n) 
\ast \dot{\mathbb{T}}(> \kappa_n)) \times \Col{(\kappa_n, \delta)}$. 
Let $(G_{> \kappa_n} \ast \dot{H}_{> \kappa_n}) \times K_{n}$
be generic for the latter poset. 

Then to show that $\mathbb{I}_\gamma(\kappa_n)$
is $\kappa_n$-Knaster in $V[G_{> \kappa_n}]$ it suffices to show that it satisfies this property in 
$V[(G_{> \kappa_n} \ast \dot{H}_{> \kappa_n}) \times K_n]$. But $(\mathbb{I}_\gamma(> \kappa_n) \ast 
\dot{\mathbb{T}}(> \kappa_n)) \times \Col{(\kappa_n, \delta)} \simeq \Col{(\kappa_n, \delta)}$ is 
$< \kappa_n$-directed closed and therefore $\kappa_n$ is supercompact (and in particular weakly compact) 
in $V[(G_{> \kappa_n} \ast \dot{H}_{> \kappa_n}) \times K_n]$. Thus $\mathbb{I}_\gamma(\kappa_n)$
is $\kappa_n$-Knaster in this generic extension by Lemma \ref{chain_condition_lem} from the proof 
of Theorem \ref{main_theorem}.
More precisely, we apply Lemma \ref{chain_condition_lem} to the pair 
$\vec{\mathbb{I}}_\gamma(\kappa_n)$, $\dot{\vec{\mathbb{I}}}_\gamma(< \kappa_n)$. 
Note that in order to do so we must have that this pair is \emph{suitable} (with regard to sufficiently closed 
elementary substructures of $H(\theta)$) in the sense of Definition \ref{suitability_defn}. 
But part (1) of this definition follows from the inductive hypothesis of the current lemma and 
part (2) follows from Lemma \ref{doesnt_add_branch_lem} in conjunction with the inductive hypothesis.

For the ``moreover'' part of the lemma, use the remark after the proof of Lemma \ref{chain_condition_lem}
rather than Lemma \ref{chain_condition_lem} itself. 

Finally, for the second part of the lemma, recall that by the inductive hypothesis we have
\begin{align*}
	V[G_{> \kappa_{n - 1}} \ast \dot{G}_{\kappa_{n - 1}}] \models \text{``$\dot{\mathbb{I}}_{\gamma}(< 
	\kappa_{n - 1})$ is $\kappa_{n - 2}$-Knaster'' }
\end{align*}
Since $G_{> \kappa_{n - 1}} = G_{> \kappa_n} \ast \, \dot{G}_{\kappa_n}$, 
$\mathbb{I}_{\gamma}(< \kappa_n) \simeq \mathbb{I}_\gamma(\kappa_{n - 1}) \, \ast \, \dot{
	\mathbb{I}}_{\gamma}(< \kappa_{n - 1})$, and 
	$V[G_{> \kappa_{n - 1}}] \models$ ``$\dot{\mathbb{I}}_\gamma(\kappa_{n - 1})$ is $\kappa_{n-1}$-Knaster''
	by the inductive hypothesis, we have 
	\begin{align*}
		V[G_{> \kappa_n} \ast \dot{G}_{\kappa_n}] = 
		V[G_{> \kappa_{n - 1}}] \models \text{$\dot{\mathbb{I}}_\gamma(< \kappa_n)$ is $\kappa_{n - 1}$-Knaster }
	\end{align*}
	as desired. 
\end{proof}

Now Theorem \ref{countably_many_theorem} follows immediately from 
Lemmas \ref{countably_many_closure_lem} and \ref{countably_many_chain_condition_lem}.

\end{proof}

Finally we may use an Easton-support iteration of the $\omega$-block posets given by Theorem \ref{countably_many_theorem} 
exactly as in \cite{golshani_hayut} to obtain a global result: 

\begin{theorem}
	Assume that there are class many supercompact cardinals with no inaccessible limit. 
	Then there is a class forcing extension of $V$ in which $\square_{\kappa, 2} + 
	\mathsf{SATP}(\kappa^+)$ holds for all regular $\kappa$. 
	\label{}
\end{theorem}

\medskip 

\begin{thebibliography}{10}

\bibitem{shelah_stanley}
Saharon Shelah and Lee Stanley.
\newblock Weakly compact cardinals and nonspecial {A}ronszajn trees.
\newblock {\em Proc. Amer. Math. Soc.}, 104(3):887--897, 1988.

\bibitem{golshani_hayut}
Mohammad Golshani and Yair Hayut.
\newblock Special aronszajn tree property, 2016.

\bibitem{baumgartner_malitz_reinhardt}
J.~Baumgartner, J.~Malitz, and W.~Reinhardt.
\newblock Embedding trees in the rationals.
\newblock {\em Proc. Nat. Acad. Sci. U.S.A.}, 67:1748--1753, 1970.

\bibitem{laver_shelah}
Richard Laver and Saharon Shelah.
\newblock The {$\aleph_{2}$}-{S}ouslin hypothesis.
\newblock {\em Trans. Amer. Math. Soc.}, 264(2):411--417, 1981.

\bibitem{jensen}
R.~Bj\"{o}rn Jensen.
\newblock The fine structure of the constructible hierarchy.
\newblock {\em Ann. Math. Logic}, 4:229--308; erratum, ibid. 4 (1972), 443,
  1972.
\newblock With a section by Jack Silver.

\bibitem{combinatorial_principles_in_core_model}
Ernest Schimmerling.
\newblock Combinatorial principles in the core model for one {W}oodin cardinal.
\newblock {\em Ann. Pure Appl. Logic}, 74(2):153--201, 1995.

\bibitem{indexed_squares}
James Cummings and Ernest Schimmerling.
\newblock Indexed squares.
\newblock {\em Israel J. Math.}, 131:61--99, 2002.

\bibitem{cc_weak_square}
Hiroshi Sakai.
\newblock Chang's conjecture and weak square.
\newblock {\em Arch. Math. Logic}, 52(1-2):29--45, 2013.

\bibitem{cummings_foreman}
James Cummings and Matthew Foreman.
\newblock The tree property.
\newblock {\em Advances in Mathematics}, 133(1):1 -- 32, 1998.

\bibitem{kunen_tall}
Kenneth Kunen and Franklin Tall.
\newblock Between martin's axiom and souslin's hypothesis.
\newblock {\em Fundamenta Mathematicae}, 102(3):173--181, 1979.

\end{thebibliography}

\end{document}